\documentclass[reqno, 12pt]{amsart}

\usepackage[margin=1in]{geometry}

\usepackage{amscd,amssymb, amsmath, wasysym}
\usepackage{graphicx}
\usepackage{amsfonts}
\usepackage{mathrsfs}    
\usepackage{amsmath}    
\usepackage{amsthm}     
\usepackage{amscd}      
\usepackage{amssymb}    
\usepackage{latexsym}   
\usepackage{graphicx}   
\usepackage{verbatim}   
\usepackage[all]{xy}     
\usepackage{xfrac}
\usepackage{amsmath,amssymb,fancyhdr,color,epsfig}

\usepackage[T2A]{fontenc}
\usepackage[utf8]{inputenc}
\usepackage[english]{babel}

\usepackage{textcase}
\usepackage{wrapfig}

\newtheorem{theorem}[subsection]{Theorem}
\newtheorem{lemma}[subsection]{Lemma}

\newtheorem{proposition}[subsection]{Proposition}
\newtheorem{corollary}[subsection]{Corollary}

\newtheorem{remark}[subsection]{Remark}
\newtheorem*{remark*}{Remark}

\newtheorem{definition}[subsection]{Definition}

\makeatletter
\@addtoreset{equation}{section}
\@addtoreset{figure}{section}
\@addtoreset{table}{section}
\makeatother

\newcommand\Mman{M}
\newcommand\Xman{X}

\newcommand\RRR{\mathbb{R}}
\newcommand\ZZZ{\mathbb{Z}}
\newcommand\FFF{\mathcal{F}}
\newcommand\id{\mathrm{id}}

\newcommand\eps{\varepsilon}

\newcommand\Diff{\mathcal{D}}
\newcommand\Stab{\mathcal{S}}
\newcommand\Orbit{\mathcal{O}}
\newcommand\DiffId{\mathcal{D}_{\id}}
\newcommand\StabId{\mathcal{S}_{\id}}

\newcommand\DiffM{\Diff(M)}
\newcommand\Stabf{\Stab(f)}
\newcommand\Orbf{\Orbit(f)}
\newcommand\Orbff{\Orbit_{f}(f)}
\newcommand\DiffIdM{\DiffId(M)}
\newcommand\StabIdf{\StabId(f)}

\newcommand{\fKRGraph}{\Gamma(f)}

\newcommand\flow{\mathbf{F}}

\newcommand\Torus{T^2}
\newcommand\torInDiff{\xi}

\newcommand\curveMeridian{C}
\newcommand\curveParallel{C'}
\newcommand\flowMeridian{\mathbf{M}}
\newcommand\flowParallel{\mathbf{L}}
\newcommand\cycleMeridian{\mu}
\newcommand\cycleParallel{\lambda}
\newcommand\subgroupMeridian{\mathcal{M}}
\newcommand\subgroupParallel{\mathcal{L}}

\newcommand\DiffIdT{\DiffId(\Torus)}
\newcommand\DiffTC{\Diff(\Torus,\curveMeridian)}
\newcommand\DiffIdTC{\DiffId(\Torus,\curveMeridian)}

\newcommand\StabPrfC{\Stab'(f,\curveMeridian)}
\newcommand\StabPrf{\Stab'(f)}

\newcommand\StabfX{\Stab(f,\Xman)}
\newcommand\StabIdfX{\StabId(f,\Xman)}
\newcommand\StabPrfX{\Stab'(f,\Xman)}
\newcommand\DiffMX{\Diff(M,\Xman)}
\newcommand\DiffIdMX{\DiffId(M,\Xman)}
\newcommand\OrbfX{\Orbit(f,\Xman)}
\newcommand\OrbffX{\Orbit_{f}(f,\Xman)}

\newcommand\SerreFibr{p}
\newcommand\Morse{\mathrm{Morse}}

\newcommand\eval{\varphi}

\newcommand\unwind{\zeta}

\newcommand\kerjo{\mathcal{K}}

\newcommand\incmap{i}
\newcommand\jZ{\incmap_0}
\newcommand\jO{\incmap_1}

\newcommand\fibr{j}

\newcommand\qhom{q} 

\newcommand\DT{\Diff^{\id}}

\newcommand\DTC{\Diff^{\id}_{\curveMeridian}}

\newcommand\Of{\Orbit}

\newcommand\OfC{\Orbit_{\curveMeridian}}

\newcommand\Sf{\Stab}

\newcommand\SfC{\Stab_{\curveMeridian}}
\newcommand\Sidf{\Stab^{\id}}
\newcommand\SidfC{\Stab^{\id}_{\curveMeridian}}

\begin{document}

\title[Orbits of smooth functions on $2$-torus and their homotopy types]
	{Orbits of smooth functions on $2$-torus and their homotopy types}

\author{Sergiy Maksymenko}
\address{Topology department, Institute of Mathematics of NAS of Ukraine, Tereshchenkivska str. 3, Kyiv, 01601, Ukraine}
\curraddr{}
\email{maks@imath.kiev.ua}

\author{Bohdan Feshchenko}
\address{Topology department, Institute of Mathematics of NAS of Ukraine, Tereshchenkivska str. 3, Kyiv, 01601, Ukraine}
\email{feshchenkobogdan@imath.kiev.ua}

\date{27/08/2014}

\begin{abstract}

Let $f:T^2\to\mathbb{R}$ be a Morse function on $2$-torus $T^2$ such that its Kronrod-Reeb graph $\Gamma(f)$ has exactly one cycle, i.e. it is homotopy equivalent to $S^1$.
Under some additional conditions we describe a homotopy type of the orbit of $f$ with respect to the action of the group of diffeomorphism of $T^2$.

This result holds  for a larger class of smooth functions $f:T^2\to\mathbb{R}$ having the following property: for every critical point $z$ of $f$ the germ of $f$ at $z$ is smoothly equivalent to a homogeneous polynomial $\mathbb{R}^2\to\mathbb{R}$ without multiple factors.

\end{abstract}


\maketitle

\section{Introduction}
Let $\Mman$ be a smooth oriented surface. 
For a closed (possibly empty) subset $\Xman \subset M$ denote by $\DiffMX$ the group of diffeomorphisms of $\Mman$ fixed on $\Xman$.
This group naturally acts from the right on the space of smooth functions $C^{\infty}(\Mman)$ by following rule:
if $h\in\DiffMX$ and $f\in C^{\infty}(\Mman)$ then the result of the action of $h$ on $f$ is the composition map
\begin{equation}\label{main-act}
f\circ h : M\xrightarrow{~~h~~} M \xrightarrow{~~f~~}\RRR.
\end{equation}
For $f\in C^{\infty}(\Mman)$ let 
\begin{align*}
\StabfX &= \{f\in\DiffMX\,|\,f\circ h = f\}, &
\OrbfX &= \{f\circ h\,|\,h\in\DiffMX\}.
\end{align*}
be respectively the \emph{stabilizer} and the \emph{orbit} of $f$ under the action~\eqref{main-act}.

Endow on $\DiffMX$, $C^{\infty}(\Mman)$ and their subspaces $\StabfX$ and $\OrbfX$ with the corresponding Whitney $C^{\infty}$-topologies.
Let also $\StabIdfX$ be the path component of the identity map $\id_{M}$ in $\StabfX$, $\DiffIdMX$ be the path component of $\id_{M}$ in $\DiffMX$, and $\OrbffX$ be the path component of $f$ in $\OrbfX$.
If $\Xman=\varnothing$ then we omit it from notation and write $\DiffM=\Diff(M,\varnothing)$, $\Stabf=\Stab(f,\varnothing)$, $\Orbf=\Orbit(f,\varnothing)$, and so on.

We will assume that all the homotopy groups of $\OrbfX$ will have $f$ as a base point, and all homotopy groups of the groups of diffeomorphisms and the corresponding stabilizers of $f$ are based at $\id_{\Mman}$.
For instance $\pi_k(\Orbit(f, \Xman))$ will always mean $\pi_k\bigl( \Orbit(f, \Xman), f \bigr)$.
Notice that the latter group is also isomorphic with $\pi_k\bigl( \Orbit_{f}(f, \Xman), f \bigr)$.

Since $\DiffMX$ and $\StabfX$ are topological groups, it follows that the homotopy sets $\pi_0\DiffMX$, $\pi_0\StabfX$, and $\pi_1\bigl(\DiffMX,\StabfX\bigr)$ have natural groups structures such that 
\[
\pi_0 \DiffMX \ \cong \ \DiffMX/\DiffIdMX,
\qquad
\pi_0 \StabfX \ \cong \ \StabfX/\StabIdfX,
\]
and in the following part of exact sequence of homotopy groups of the pair $\bigl(\DiffMX,\StabfX\bigr)$
\begin{equation}\label{equ:exact_seq_of_pair_DMX_SfX}
\cdots \to \pi_1\DiffMX \xrightarrow{~~\qhom~~} \pi_1\bigl(\DiffMX,\StabfX\bigr) 
 \xrightarrow{~~\partial~~} \pi_0 \StabfX \xrightarrow{~~\incmap~~} \pi_0 \DiffMX
\end{equation}
all maps are homomorphisms.

Moreover, $\qhom\bigl(\pi_1\DiffMX \bigr)$ is contained in the center of $\pi_1\bigl(\DiffMX,\StabfX\bigr)$.

\medskip

Recall that two smooth germs $f, g: (\RRR^2,0)\longrightarrow (\RRR,0)$ are said to be \emph{smothly equivalent} if there exist germs of diffeomorphisms $h:(\RRR^2,0)\longrightarrow(\RRR^2,0)$ and $\phi:(\RRR,0)\to (\RRR,0)$ such that $\phi\circ g = f\circ h$.

\begin{definition}
Denote by $\mathcal{F}(M)$ a subset in $C^{\infty}(\Mman)$ which consists of functions $f$ having the following two properties:
\begin{itemize}
\item
$f$ takes a constant value at each connected components of $\partial\Mman$, and all critical points of $f$ are contained in the interior of $M$;

\item for each critical point $z$ of $f$ the germ of $f$ at $z$ is smoothly equivalent to a {\bfseries homogeneous polynomial $f_z:\RRR^2\to\RRR$ without multiple factors}.
\end{itemize}
\end{definition}

Suppose a smooth germ $f: (\RRR^2,0)\longrightarrow (\RRR,0)$ has a critical point $0\in\RRR^2$.
This point is called \emph{non-degenerate} if $f$ is smoothly equivalent to a homogeneous polynomial of the form $\pm x^2 \pm y^2$.

Denote by $\Morse(M)$ the subset of $C^{\infty}(\Mman)$ consisting of Morse functions, that is functions having only \textit{non-degenerate} critical points.
It is well known that $\Morse(M)$ is open and everywhere dense in $C^{\infty}(\Mman)$.
Since $\pm x^2 \pm y^2$ has no multiple factors, we get the following inclusion $\Morse(M) \ \subset \ \FFF(\Mman)$.

\begin{remark}\rm
A homogeneous polynomial $f:\RRR^2\to\RRR$ has critical points only if $\deg f \geq2$, and in this case the origin is always a critical point of $f$.
If $f$ has no multiple factors, then the origin $0$ is a unique critical point.
Moreover, $0$ is non-degenerate $\deg f = 2$, and degenerate for $\deg f \geq 3$, see~\cite[\S7]{Maksymenko:MFAT:2009}.
\end{remark}

Now let $f\in \FFF(\Mman)$ and $c\in\RRR$.
A connected component $\curveMeridian$ of the level set $f^{-1}(c)$ is said to be \emph{critical} if $\curveMeridian$ contains at least one critical point of $f$.
Otherwise $\curveMeridian$ is called \emph{regular}.
Consider a partition of $\Mman$ into connected component of level sets of $f$.
It is well known that the corresponding factor-space $\fKRGraph$ has a structure of a finite one-dimensional $CW$-complex and is called \emph{Kronrod-Reeb graph} or simply KR-graph of the function $f$.
In particular, the vertices of $\fKRGraph$ are critical components of level sets of $f$.

It is usually said that this graph was introduced by G.~Reeb in~\cite{Reeb:ASI:1952}, however in was used before by A.~S.~Kronrod in~\cite{Kronrod:UMN:1950} for studying functions on surfaces.
Applications of $\fKRGraph$ to study Morse functions on surfaces are given e.g. in~\cite{BolsinovFomenko:1997, Kulinich:MFAT:1998, Kudryavtseva:MatSb:1999, Sharko:UMZ:2003, Sharko:MFAT:2006, MasumotoSaeki:KJM:2011}.

\medskip

In a series of papers the first author calculated homotopy types of spaces $\Stabf$ and $\Orbf$ for all $f \in \FFF(\Mman)$.
These results are summarized in Theorem~\ref{th:fibration_DMX_Of} below.

Denote also
\begin{align*}
\StabPrf &= \Stabf \cap \DiffIdM,
&
\StabPrfX &= \Stabf \cap \DiffIdMX.
\end{align*}
Thus $\StabPrfX$ consists of diffeomorphisms $h$ preserving $f$, fixed on $\Xman$ and isotopic to $\id_{M}$ relatively $\Xman$, though the isotopy between $h$ and $\id_{M}$ is not required to be $f$-preserving.

\begin{theorem}\label{th:fibration_DMX_Of}
{\rm \cite{Maksymenko:AGAG:2006, Maksymenko:ProcIM:ENG:2010, Maksymenko:UMZ:ENG:2012}.}
Let $f\in\FFF(\Mman)$ and $\Xman$ be a finite (possibly empty) union of regular components of certain level sets of function $f$.
Then the following statements hold true.

{\rm (1)}
$\Orbit_{f}(f, \Xman) = \Orbit_{f}(f,\Xman\cup\partial\Mman)$, and so 
\[
\pi_k \Orbit(f, \Xman) \ \cong \ \pi_k \Orbit(f, \Xman \cup\partial\Mman), \qquad k\geq1.
\]

{\rm (2)}
The following map
\[
 \SerreFibr:\DiffMX \longrightarrow \OrbfX, \qquad \SerreFibr(h) = f \circ h.
\]
is a Serre fibration with fiber $\StabfX$, i.e.\! it has homotopy lifting property for CW-complexes.
This implies that 
\begin{itemize}
\item[\rm(a)]
$\SerreFibr(\DiffIdMX) = \OrbffX$;

\item[\rm(b)]
the restriction map 
\begin{equation}\label{equ:fibr_pX}
\SerreFibr|_{\DiffIdMX}:\DiffIdMX \longrightarrow \OrbffX
\end{equation}
is also a Serre fibration with fiber $\StabPrfX$;

\item[\rm(c)]
for each $k\geq0$ we have an isomorphism $\fibr_k: \pi_k \bigl( \DiffMX, \StabfX \bigr) \longrightarrow \pi_k \OrbfX$ defined by $\fibr_k[\omega] = [f\circ\omega]$ for a continuous map $\omega:(I^k, \partial I^k, 0) \to \bigl( \DiffM, \Stabf, \id_{\Mman} \bigr)$, and making commutative the following diagram
\[
\xymatrix{ 
\cdots \ar[r] & \pi_k\DiffMX \ar[r]^-{\qhom} \ar[rd]_{p} & \pi_k\bigl(\DiffMX,\StabfX\bigr)\ar[d]^{\fibr_k}_{\cong} \ar[r]^-{\partial} & \pi_{k-1} \StabfX \ar[r] & \cdots \\
 & & \pi_k\OrbfX \ar[ru]_{\partial \circ \fibr_k^{-1}},
}
\]
see for example \cite[\S~4.1, Theorem~4.1]{Hatcher:AlgTop:2002}.
\end{itemize}

{\rm (3)}
Suppose either $f$ has a critical point which is not a {\bfseries nondegenerate local extremum} or $\Mman$ is a non-oriented surface.
Then $\StabIdf$ is contractible, $\pi_n\Orbf = \pi_n\Mman$ for $n\geq3$, $\pi_2\Orbf=0$, and for $\pi_1\Orbf$ we have the following short exact sequence of fibration $\SerreFibr$:
\begin{equation}\label{equ:pi1Of_exact_sequence}
 1 \longrightarrow \pi_1\DiffM \xrightarrow{~~\SerreFibr~~} \pi_1\Orbf \xrightarrow{~~\partial\circ \fibr^{-1}_1~~} \pi_0\StabPrf\longrightarrow 1.
\end{equation}
Moreover, $\SerreFibr\bigl(\pi_1\DiffM\bigr)$ is contained in the center of $\pi_1\Orbf$.

{\rm (4)}
Suppose either $\chi(\Mman)<0$ or $\Xman\not=\varnothing$.
Then $\DiffIdMX$ and $\StabIdfX$ are contractible, whence from the exact sequence of homotopy groups of the fibration~\eqref{equ:fibr_pX} we get $\pi_k\OrbfX=0$ for $k\geq2$, and that the boundary map 
\[ \partial\circ \fibr^{-1}_1: \pi_1\OrbfX \ \longrightarrow \ \pi_0 \StabPrfX \]
is an isomorphism.
\end{theorem}

Suppose $\Mman$ is orientable and differs from the sphere $S^2$ and the torus $\Torus$, and let $\Xman = \partial\Mman$.
Then $\Mman$ and $\Xman$ satisfy assumptions of (4) of Theorem~\ref{th:fibration_DMX_Of}.
Therefore from (1) of that theorem we get the following isomorphism
\[
\pi_1\Orbf \ \stackrel{(1)}{\cong} \ \pi_1\Orbit(f, \partial\Mman) \ \stackrel{(4)}{\cong} \ \pi_0 \Stab'(f,\partial\Mman).
\]
A possible structure of $\pi_0\Stab'(f,\partial\Mman)$ for this case is completely described in~\cite{Maksymenko:pi1Repr:2014}.

However when $\Mman$ is a sphere or a torus the situation is more complicated, as $\pi_1\DiffM \not=0$ and from the short exact sequence~\eqref{equ:pi1Of_exact_sequence} we get only that $\pi_1\Orbf$ is an extension of $\pi_0\StabPrf$ with $\pi_1\DiffM$.

\section{Main result}
Suppose $\Mman = \Torus$.
Then it can easily be shown that for each $f\in\FFF(\Torus)$ its KR-graph $\fKRGraph$ is either a tree or has exactly one simple cycle.
Moreover, $\pi_1\DT = \ZZZ^2$, see~\cite{EarleEells:BAMS:1967, Gramain:ASENS:1973}, and therefore the sequence \eqref{equ:pi1Of_exact_sequence} can be rewritten as follows:
\begin{equation}\label{equ:T2_pi1Of_exact_sequence}
1 \longrightarrow \ZZZ^2 \xrightarrow{~~p~~} \pi_1\Orbff \xrightarrow{~~\partial~~} \pi_0 \StabPrf \longrightarrow 1.
\end{equation}

In~\cite{MaksymenkoFeshchenko:UMZ:ENG:2014} the authors studied the case when $\fKRGraph$ is a tree and proved that under certain ``triviality of $\StabPrf$-action'' assumptions on $f$ the sequence~\eqref{equ:T2_pi1Of_exact_sequence} splits and we get an isomorphism
$\pi_1\Orbff\cong\pi_0\StabPrf\times \ZZZ^{2}$.

In the present paper we consider the situation when $\fKRGraph$ is has exactly one simple cycle $\Upsilon$ and under another ``triviality of $\StabPrf$-action'' assumption describe the homotopy type of $\Orbff$ in terms of $\pi_0\StabPrfC$ for some regular component of some level-set of $f$, see Definition~\ref{def:triv_act_S_on_C} and Theorem~\ref{th:main:pi1Of}.

First we mention the following two simple lemmas which are left for the reader.
\begin{lemma}\label{lm:characterization_KRGRaph_tree}
Let $f\in\FFF(\Torus)$.
Then the following conditions are equivalent:
\begin{itemize}
\item[\rm(i)] $\fKRGraph$ is a tree;
\item[\rm(ii)] every point $z\in\fKRGraph$ separates $\fKRGraph$;
\item[\rm(iii)] every connected component of every level set of $f$ separates $\Torus$.
\end{itemize}
\end{lemma}

\begin{lemma}\label{lm:characterization_of_cycle}
Assume that $\fKRGraph$ has exactly one simple cycle $\Upsilon$.
Let also $z\in\fKRGraph$ be any point belonging to some open edge of $\fKRGraph$ and $\curveMeridian$ be the corresponding regular component of certain level set $f^{-1}(c)$ of $f$.
Then the following conditions are equivalent:
\begin{itemize}
\item[\rm(a)] $z\in\Upsilon$;
\item[\rm(b)] $z$ does not separate $\fKRGraph$;
\item[\rm(c)] $\curveMeridian$ does not separate $\Torus$.
\end{itemize}
\end{lemma}

Thus if $\fKRGraph$ is not a tree, then there exists a connected component $\curveMeridian$ of some level set of $f$ that does not separate $\Torus$, and this curve corresponds to some point $z$ on an open edge of cycle $\Upsilon$.

For simplicity we will fix once and for all such $f\in\FFF(\Torus)$ and $\curveMeridian$, and use the following notation:
\begin{align*}
\DT &:= \DiffId(\Torus), &
\Of &:= \Orbit_f(f), &
\Sf &:= \Stab'(f), &
\Sidf &:= \StabId(\Torus), \\
\DTC &:= \DiffId(\Torus, \curveMeridian), &
\OfC &:= \Orbit_f(f,\curveMeridian), &
\SfC &:= \Stab'(f,\curveMeridian), &
\SidfC &:= \StabId(f,\curveMeridian)
\end{align*}

Let also $h \in \Sf$, so $f\circ h = f$ and $h$ is isotopic to $\id_{\Torus}$.
Then $h(f^{-1}(c)) = f^{-1}(c)$, and therefore $h$ interchanges connected components of $f^{-1}(c)$.
In particular, $h(\curveMeridian)$ is also a connected component of $f^{-1}(c)$.
However, in general, $h(\curveMeridian)$ does not coincide with $\curveMeridian$, see Figure~\ref{fig:hC_not_C}.
\begin{figure}[h]
\center{\includegraphics[width=0.6\linewidth]{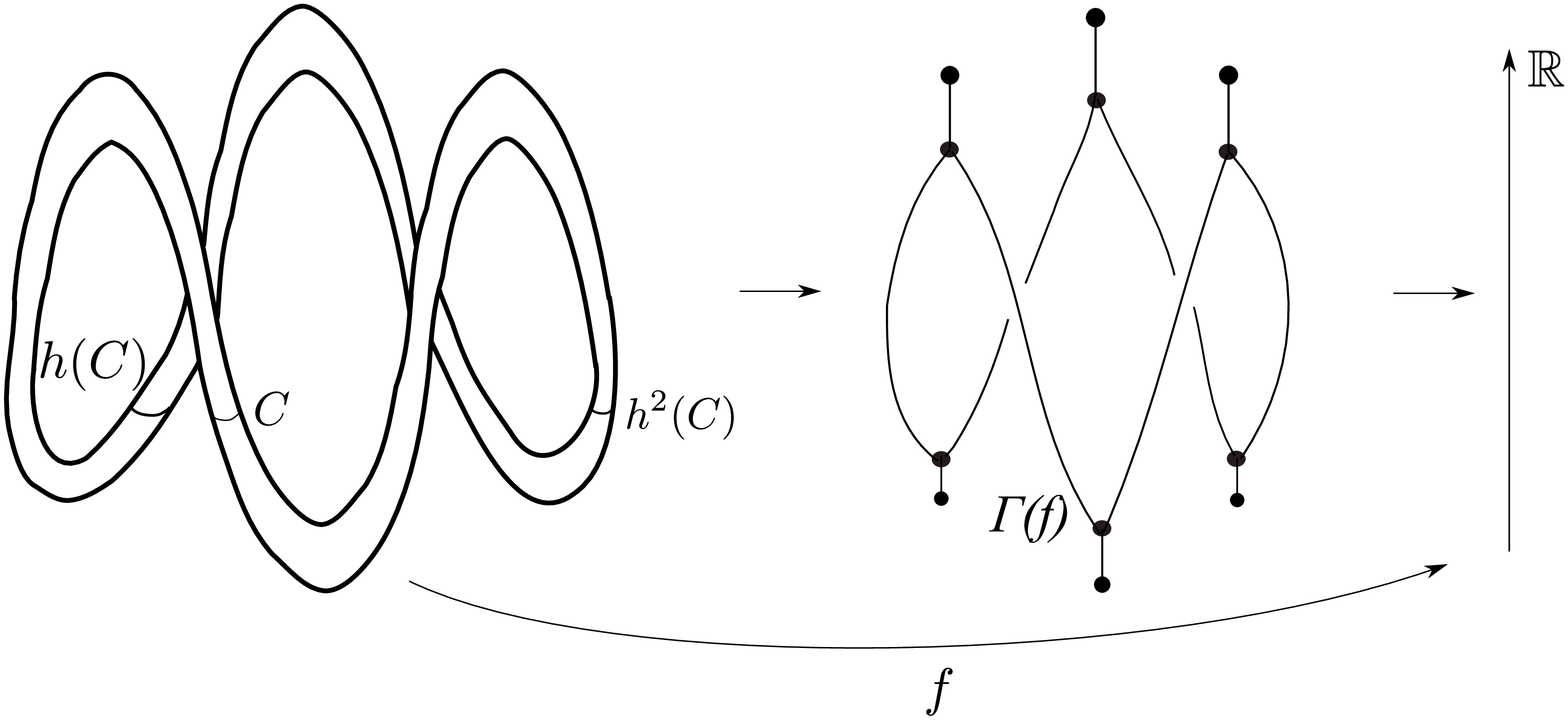}}
\caption{}
\label{fig:hC_not_C}
\end{figure}

\begin{definition}\label{def:triv_act_S_on_C}
We will say that \textit{$\Sf$ trivially acts on $\curveMeridian$} if $h(\curveMeridian) = \curveMeridian$ for all $h\in\Sf$.
\end{definition}

\begin{remark}\rm
Emphasize that the above definition only require that $h(\curveMeridian)=\curveMeridian$ for all diffeomorphisms $h$ that preserve $f$ and are \textit{isotopic to $\id_{\Torus}$}.
We do not put any assumptions on diffeomorphisms that are not isotopic to $\id_{\Torus}$.
\end{remark}

\begin{remark}\rm
It is easy to see that if $\curveMeridian_1$ is another non-separating regular component of some level-set $f^{-1}(c_1)$, then $\Sf$ trivially acts on $\curveMeridian$ if and only if $\Sf$ trivially acts on $\curveMeridian_1$.
\end{remark}

\begin{theorem}\label{th:main:pi1Of}
Let $f\in\FFF(\Torus)$ be such that $\fKRGraph$ has exactly one cycle, and $\curveMeridian$ be a regular connected component of certain level set $f^{-1}(c)$ of $f$ which does not separate $\Torus$.
Suppose $\Sf$ trivially acts on $\curveMeridian$.
Then there is a homotopy equivalence
$$
\Of \simeq \OfC\times S^1.
$$
In particular, we have the following isomorphisms:
$$
\pi_1\Of\ \cong \ \pi_1 \OfC\times\ZZZ  \ \stackrel{\fibr\times\id_{\ZZZ}}{\cong}\ \pi_0\SfC\times\ZZZ.
$$
\end{theorem}

\noindent
The proof of this theorem will be given in \S~\ref{sect:proof:th:main:pi1Of}.

\section{Preliminaries}\label{sect:prelim}

\subsection{Algebraic lemma}
\begin{lemma}\label{lm:partial_splitting}
Let $L, M, Q, S$ be four groups.
Suppose there exists a short exact sequence
\begin{equation}\label{equ:alg_lemma_exact_seq}
1 \to L \times M  \xrightarrow{~~\qhom~~} T \xrightarrow{~~\partial~~} S \to 1
\end{equation}
and a homomorphism $\eval: T \to L \times 1$ such that 
\begin{itemize}
\item
$\eval \circ ~\qhom: L\times 1 \longrightarrow L \times 1$ is the identity map and
\item
$\qhom(1 \times M) \subset \ker(\eval)$.
\end{itemize}
Then we have the following exact sequence:
\[
1 \longrightarrow 1\times M  \xrightarrow{~~q~~} \ker(\eval) \xrightarrow{~~\partial~~} S \longrightarrow 1.
\]
\end{lemma}
\begin{proof}
It suffices to prove that
\begin{itemize}
\item[1)] $\partial(\ker\eval) = S$,
\item[2)] $\eval(1\times M) = \ker\eval \cap \ker\partial$.
\end{itemize}

1) Let $s \in S$.
We have for find $b\in \ker\eval$ such that $\partial(b) = s$.
Since $\partial(T) = S$, there exists $t \in T$ such that $\partial(t) = s$.
Put $\hat{t} = \qhom(\eval(t))$ and $b = t\,\hat{t}^{-1}$.
Then
\begin{align*}
\eval(\hat{t}) &= \eval \circ \qhom \circ \eval(t) = \eval(t),
&
\partial(\hat{t}) &= \partial \circ \qhom \circ \eval(t)  = 1.
\end{align*}
Hence $b= t \hat{t}^{-1} \in \ker\eval$ and 
\[ 
\partial(b) = \partial(t) \,\partial(\hat{t})^{-1} = \partial(t) = s.
\]

2) Let $a \in \ker\eval \cap \ker\partial$.
We should find $m \in M$ such that $\qhom(1,m) = a$.

As $a \in \ker\partial = \qhom(L \times M)$, so there exist $(l,m) \in L\times M$ such that $\qhom(l,m) = a$.
But $\qhom(1,m) \in \qhom(1\times M) \subset \ker\eval$, whence
\[
(1,1) = \eval(a) = \eval(\qhom(l,m)) = \eval(\qhom(l,1)) \cdot  \eval(\qhom(1,m)) =  \eval(\qhom(l,1)) = (l,1).
\]
Hence $l=1$, and so $a = \qhom(1,m) \in \qhom(1\times M)$. 
\end{proof}

\subsection{Isotopies of $\Torus$ fixed on a curve}
We will need the following general lemma claiming that if a diffeomorphism $h$ of $\Torus$ is fixed on a non-separating simple closed curve $\curveMeridian$ and is isotopic to $\id_{\Torus}$, then an isotopy between $h$ and $\id_{\Torus}$ can be made fixed on $\curveMeridian$.

\begin{lemma}\label{lm:h_C_id__h_in_DidTC}
Let $\curveMeridian \subset \Torus$ be a not null-homotopic smooth simple closed curve.
Then
\begin{equation}\label{equ:DiffIdT_cap_DiffTC__DiffIdTC}
 \DiffIdTC = \DiffIdT\cap \DiffTC.
\end{equation}
\end{lemma}
\begin{proof}
The inclusion $\DiffIdTC \ \subset \ \DiffIdT\,\cap\, \DiffTC$ is evident.
Therefore we have to establish the inverse one.

Let $h\in\DiffIdT\,\cap\, \DiffTC$, so $h$ is fixed on $\curveMeridian$ and is isotopic to $\id_{\Torus}$.
We have to prove that $h\in \DiffIdTC$, i.e. it is isotopic to $\id_{\Torus}$ via an isotopy fixed on $\curveMeridian$.

Let $\curveMeridian_1 \subset \Torus$ be a simple closed curve isotopic to $\curveMeridian$ and disjoint from $\curveMeridian$, and $\tau:\Torus\to\Torus$ be a Dehn twist along $\curveMeridian_1$ fixed on $\curveMeridian$.
Cut the torus $\Torus$ along $\curveMeridian$ and denote the resulting cylinder by $Q$.

Notice that the restrictions $h|_{Q}, \tau|_{Q}: Q \to Q$ are fixed on $\partial Q$.
It is well-known that the isotopy class $\tau|_{Q}$ generates the group $\pi_0\Diff(Q, \partial Q) \cong\ZZZ$.
Hence there exists $n\in\ZZZ$ such that $h|_{Q}$ is isotopic to $\tau^n|_{Q}$ relatively to $\partial Q$.
This isotopy induces an isotopy between $h$ and $\tau^n$ fixed on $\curveMeridian$.

By assumption $h$ is isotopic to $\id_{\Torus}$, while $\tau^{n}$ is isotopic to $\id_{\Torus}$ only for $n=0$.
Hence $h$ is isotopic to $\tau^0 = \id_{\Torus}$ via an isotopy fixed of $\curveMeridian$.
\end{proof}

\subsection{Smooth shifts along trajectories of a flow}
Let $\flow:M \times \RRR\to M$ be a smooth flow on a manifold $M$.
Then for every smooth function $\alpha:M \to \RRR$ one can define the following map $\flow_{\alpha}:\Torus\to\RRR$ by the formula:
\begin{equation}\label{equ:shooth_shift}
\flow_{\alpha}(z) = \flow(z, \alpha(z)), \qquad z\in M. 
\end{equation}

\begin{lemma}
If $\flow_{\alpha}$ is a {\bfseries diffeomorphism} then for each $t\in[0,1]$ the map
\[
\flow_{t\alpha}:M \to M,
\qquad 
\flow_{t\alpha}(z) = \flow(z, t\alpha(z))
\]
is a diffeomorphism as well.
In particular, $\{\flow_{t\alpha}\}_{t\in I}$ is an isotopy between $\id_{M} = \flow_0$ and $\flow_{\alpha}$.
\end{lemma}

\subsection{Some constructions associated with $f$}\label{sect:some_constructions}
In the sequel we will regard the circle $S^1$ and the torus $\Torus$ as the corresponding factor-groups $\RRR/\ZZZ$ and $\RRR^2/\ZZZ^2$.
Let $e = (0,0) \in \Torus$ be the unit of $\Torus$.
We will always assume that $e$ is a base point for all homotopy groups related with $\Torus$ and its subsets.
For $\eps\in(0,0.5)$ let \[ J_{\eps} = (-\eps,\eps) \subset S^1\]
be an open $\eps$-neighbourhood of $0\in S^1$.

\medskip 

Let $f\in\FFF(\Torus)$ be a function such that its KR-graph $\fKRGraph$ has only one cycle, and let $\curveMeridian$ be a regular connected component of certain level set of $f$ not separating $\Torus$.
For this situation we will now define several constructions ``adopted'' with $f$.

\medskip 

{\bf Special coordinates.}
Since $\curveMeridian$ is non-separating and if a regular component of $f^{-1}(c)$, one can assume (by a proper choice of coordinates on $\Torus$) that the following two conditions hold:
\begin{enumerate}
\item[\rm(a)]
$\curveMeridian = 0 \times S^1 \ \subset \ \RRR^2/\ZZZ^2 \equiv \Torus$\,;
\item[\rm(b)] 
there exists $\eps>0$ such that for all $t \in J_{\eps} = (-\eps,\eps)$ the curve $t \times S^1$ is a regular connected component of some level set of $f$.
\end{enumerate}
It is convenient to regard $\curveMeridian$ as a \emph{meridian} of $\Torus$.
Let $\curveParallel = S^1 \times 0$ be the corresponding \emph{parallel}.
Then $\curveParallel \cap \curveMeridian  = e$, see Figure~\ref{fig:spec_coordinates}.
Consider also the following loops $\cycleParallel,\cycleMeridian: I \to \Torus$ defined by 
\begin{align}\label{equ:loops_paral_merid}
\cycleParallel(t) &= (t, 0), & \cycleMeridian(t) &= (0, t).
\end{align}
They represent the homotopy classes of $\curveParallel$ and $\curveMeridian$ in $\pi_1\Torus$ respectively.

Let us also mention that $\curveMeridian$ is a \textit{subgroup} of the group $\Torus$.
Therefore $\pi_1(\Torus,\curveMeridian)$ has a natural groups structure.

Let $k: \curveMeridian \hookrightarrow \Torus$ be the inclusion map.
Then the corresponding homomorphism $k:\pi_1\curveMeridian\to\pi_1\Torus$ is injective.
Since $\curveMeridian$ is also connected, i.e.\! $\pi_0\curveMeridian = \{1\}$, we get the following short exact sequence of homotopy groups of the pair $(\Torus,\curveMeridian)$:
\begin{equation}\label{equ:exact_seq_T_C}
1 \longrightarrow \pi_1\curveMeridian  \xrightarrow{~~k~~}  \pi_1\Torus  \xrightarrow{~~r~~} \pi_1(\Torus,\curveMeridian) \longrightarrow 1.
\end{equation}
As $\pi_1\curveMeridian \cong\ZZZ$ and $\pi_1\Torus\cong\ZZZ^2$, it follows that $\pi_1(\Torus,\curveMeridian) \cong \ZZZ$ and this group is generated by the image $r[\cycleParallel]$ of the homotopy class of the parallel $\cycleParallel$.
In particular, there exists a section 
\begin{equation}\label{equ:sect_of_r}
s: \pi_1(\Torus,\curveMeridian) \longrightarrow \pi_1\Torus
\end{equation}
such that $r \circ s[\cycleParallel] = [\cycleParallel]$, so $r\circ s$ is the identity map of $\pi_1(\Torus,\curveMeridian)$.

\begin{figure}[h]
\center{\includegraphics[width=6cm]{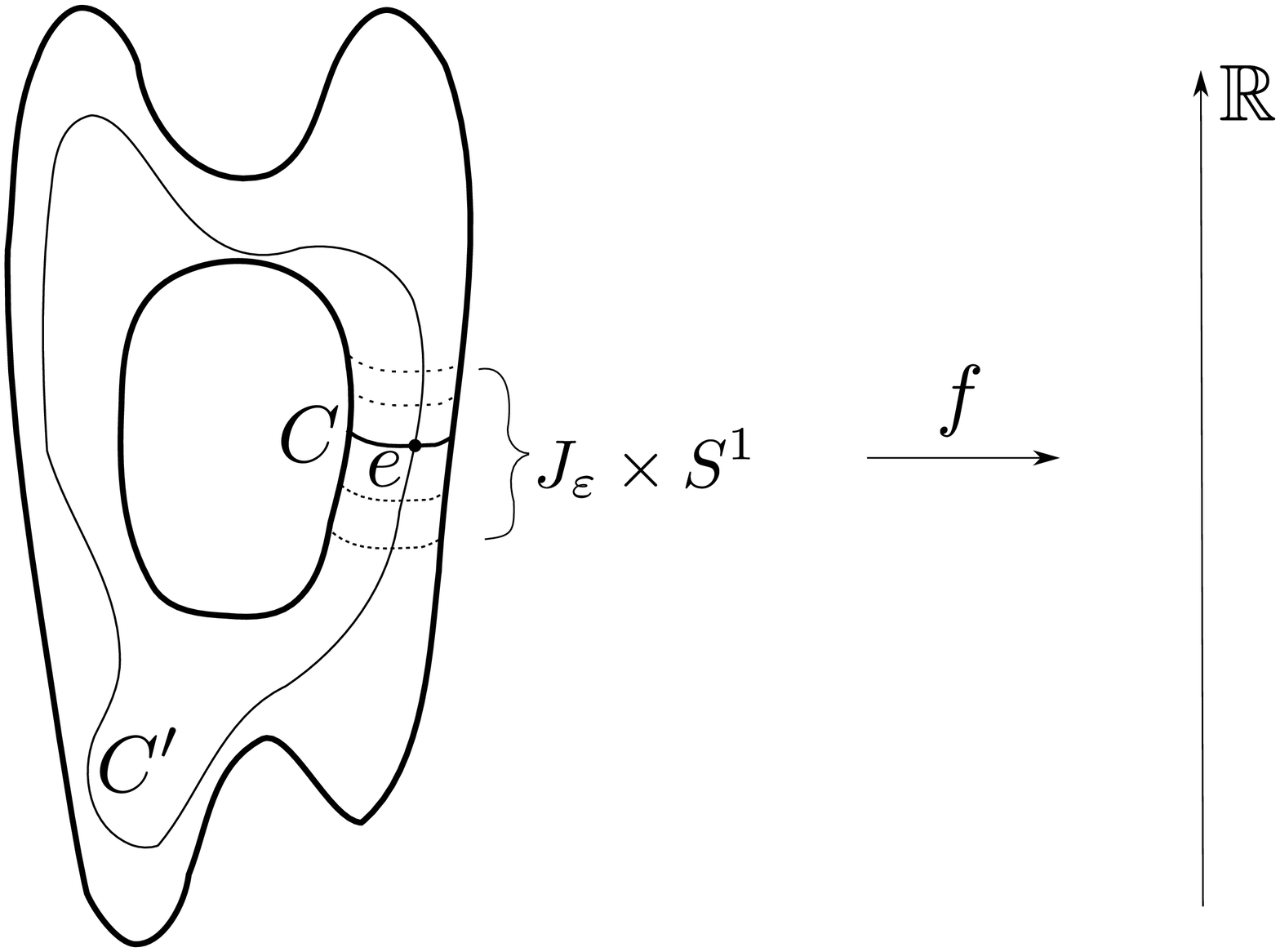}}
\caption{}
\label{fig:spec_coordinates}
\end{figure}
\bigskip

{\bf An inclusion $\torInDiff:\Torus \subset \DT$. }
Notice that $\Torus$ is a connected Lie group.
Therefore it acts on itself by smooth left translations.
This yields the following embedding $\torInDiff:\Torus\hookrightarrow\DT$: if $(a,b)\in\Torus$, then $\torInDiff(a,b):\Torus\to\Torus$ is a diffeomorphism given by the formula:
\begin{equation}\label{equ:inclusion_T2_DT2}
\torInDiff(a,b)(x,y) = \bigl(x+a \ \mathrm{mod} \ 1, \ \ y+b \ \mathrm{mod} \ 1\bigr).
\end{equation}
It is well known that $\torInDiff$ is a homotopy equivalence, see e.g.~\cite{Gramain:ASENS:1973}.

Notice also that $\torInDiff$ yields the following map
\begin{equation}\label{equ:pi1T2_cont_maps_ind_map}
\torInDiff: C\bigl( (I, \partial I), (\Torus,e) \bigr)
\longrightarrow 
C\bigl( (I, \partial I), (\DT, \id_{\Torus}) \bigr)
\end{equation}
between the \textit{spaces of loops} defined as follows: if $\omega:(I, \partial I) \longrightarrow (\DT, \id_{\Torus})$ is a continuous map, then 
\[
\torInDiff(\omega) = \torInDiff \circ \omega: \ I \longrightarrow \DT.
\]
It is well known that this map is continuous with respect to compact open topologies.
Moreover the corresponding map between the path components is just the homomorphism of fundamental groups:
\begin{equation}\label{equ:iso_pi1T2_pi1DT2}
 \torInDiff: \pi_1\Torus \longrightarrow \pi_1\DT.
\end{equation}
Since $\torInDiff$ is a homotopy equivalence, the homomorphism~\eqref{equ:iso_pi1T2_pi1DT2} is in fact an isomorphism.

To simplify notation we denoted all these maps with the same letter $\torInDiff$.
However this will never lead to confusion.

\medskip 

{\bf Isotopies $\flowParallel$ and $\flowMeridian$.}
Let 
\begin{align}\label{equ:LM_loops}
\flowParallel &= \torInDiff(\cycleParallel), 
&
\flowMeridian &= \torInDiff(\cycleMeridian) 
\end{align}
be the images of the loops $\cycleParallel$ and $\cycleMeridian$ in $\DT$ under the map Eq.~\eqref{equ:pi1T2_cont_maps_ind_map}
Evidently, they can be regarded as isotopies $\flowParallel, \flowMeridian: \Torus \times [0,1]\to \Torus$ defined by
\begin{align}\label{equ:LM_flows}
\flowParallel(x,y,t) &= (x + t \ \mathrm{mod} \ 1,\  y),
&
\flowMeridian(x,y,t) &= (x, \ y + t \ \mathrm{mod} \ 1),
\end{align}
for $x\in \curveParallel$, $y\in \curveMeridian$, and $t\in[0,1]$.
Geometrically, $\flowParallel$ is a ``\emph{rotation}'' of the torus along its parallels and $\flowMeridian$ is a rotation along its meridians.

Denote by $\subgroupParallel$ and $\subgroupMeridian$ the subgroups of $\pi_1\DT$ generated by loops $\flowParallel$ and $\flowMeridian$ respectively.
Since $\torInDiff:\Torus \to \DT$ is a homotopy equivalence, and the loops $\cycleParallel$ and $\cycleMeridian$ freely generate $\pi_1\Torus$, it follows that $\subgroupParallel$ and $\subgroupMeridian$ are commuting free cyclic groups, and so we get an isomorphism:
\[
\pi_1\DT \cong \subgroupParallel \times \subgroupMeridian.
\]

Also notice that $\flowParallel$ and $\flowMeridian$ can be also regarded as \emph{flows} $\flowParallel, \flowMeridian: \Torus \times \RRR\to \Torus$ defined by the same formulas Eq.~\eqref{equ:LM_flows} for $(x,y,t)\in\Torus\times\RRR$.
All orbits of the \textit{flows} $\flowParallel$ and $\flowMeridian$ are periodic of period $1$.
We will denoted these flows by the same letters as the corresponding \textit{loops}~\eqref{equ:LM_loops}, however this will never lead to confusion.

\bigskip

{\bf A flow $\flow$.}
As $\Torus$ is an orientable surface, there exists a flow $\flow:\Torus\times\RRR\to\Torus$ having the following properties, see e.g.~\cite[Lemma~5.1]{Maksymenko:AGAG:2006}:
\begin{itemize}
\item[1)]
a point $z\in\Torus$ is fixed for $\flow$ if and only if $z$ is a critical point of $f$;
\item[2)]
$f$ is constant along orbits of $\flow$, that is $f(z) = f(\flow(z,t))$ for all $z\in\Torus$ and $t\in\RRR$.
\end{itemize}
It follows that every critical point of $f$ and every regular components of every level set of $f$ is an orbit of $\flow$.

In particular, each curve $t\times S^1$ for $t\in J_{\eps}$ is an orbit of $\flow$.
On the other hand, this curve is also an orbit of the flow $\flowMeridian$.
Therefore, we can always choose $\flow$ so that
\begin{equation}\label{equ:relation_flowMer_flowHam}
\flowMeridian(x,y,t) = \flow(x,y,t),
\qquad
(x,y,t) \in  J_{\eps}\times S^1\times\RRR.
\end{equation}

\begin{lemma}\label{lm:shift_functions}{\rm\cite[Lemma~5.1]{Maksymenko:AGAG:2006}.}
Suppose a flow $\flow:\Torus\times\RRR\to\Torus$ satisfies the above conditions {\rm1)} and {\rm2)} and let $h\in\Stabf$.
Then $h\in\StabIdf$ if and only if there exists a $C^{\infty}$ function $\alpha:\Torus\to\RRR$ such that $h = \flow_{\alpha}$, see Eq.~\eqref{equ:shooth_shift}.
Such a function is unique and the family of maps $\{\flow_{t\alpha}\}_{t\in I}$ constitute an isotopy between $\id_{M}$ and $h$.
\qed
\end{lemma}

\section{Proof of Theorem~\ref{th:main:pi1Of}}\label{sect:proof:th:main:pi1Of}
Let $f\in\FFF(\Torus)$ be such that its KR-graph $\fKRGraph$ has only one cycle, and let $\curveMeridian$ be a non-separating regular connected component of certain level set of $f$.
Assume also that $\Sf$ trivially acts of $\curveMeridian$.
We have to prove that there exists a homotopy equivalence $\OfC\times S^1 \simeq \Of$.

By (3) and (4) of Theorem~\ref{th:fibration_DMX_Of} the orbits $\Of$ and $\OfC$ are aspherical, as well as $S^1$, i.e. their homotopy groups $\pi_k$ vanish for $k\geq2$.
Therefore, by Whitehead Theorem~\cite[\S~4.1, Theorem~4.5]{Hatcher:AlgTop:2002}, it suffices to show that there exists an isomorphism
$\pi_1\OfC \times \pi_1 S^1 \ \cong \ \pi_1\Of$.
Such an isomorphism will induce a required homotopy equivalence. 

Moreover, due to (2) of Theorem~\ref{th:fibration_DMX_Of} we have isomorphisms: 
\begin{align*}
\pi_1(\DTC, \SfC) & \ \cong \ \pi_1 \OfC, &
\pi_1(\DT, \Sf) & \ \cong \ \pi_1 \Of.
\end{align*}
Therefore it remains to find the following isomorphism:
\begin{equation}\label{equ:pi1DS_pi1DCSC_Z}
\pi_1(\DTC, \SfC) \times \pi_1 S^1 \ \cong \ \pi_1(\DT, \Sf).
\end{equation}

Notice that every smooth function $f:\Torus\to\RRR$ always have critical points being not local extremes, since otherwise $\Torus$ would be diffeomorphic with a $2$-sphere $S^2$.
Therefore by (3) and (4) of Theorem~\ref{th:main:pi1Of} the spaces $\Sf$, $\SfC$, and $\DTC$ are contractible.
Moreover, as noted above, $\DT$ is homotopy equivalent to $\Torus$.

Let $i:(\DTC, \SfC) \subset (\DT, \Sf)$ be the inclusion map.
It yields a morphism between the exact sequences of homotopy groups of these pairs.
The non-trivial part of this morphism is contained in the following commutative diagram:
\begin{equation}\label{equ:morphism_of_sequences}
\begin{CD}
1 @>{}>> 1 @>>> \pi_1(\DTC,\SfC) @>{\partial_{\curveMeridian}}>> \pi_0\SfC @>{}>> 1 \\
&& @V{}VV @V{\jO}VV @V{\jZ}VV \\
1 @>{}>> \pi_1\DT @>{\qhom}>> \pi_1(\DT,\Sf) @>{\partial}>> \pi_0\Sf @>{}>> 1
\end{CD}
\end{equation}
The proof of Theorem~\ref{th:main:pi1Of} is based on the following two Propositions~\ref{pr:hom_onto_L} and~\ref{pr:j0_properties} below.

\begin{proposition}\label{pr:hom_onto_L}
Under assumptions of Theorem~\ref{th:main:pi1Of} there exists an epimorphism 
\[ \eval:\pi_1(\DT, \Sf)\longrightarrow \subgroupParallel \]
such that 
\begin{enumerate}
\item[\rm1)]
$\eval$ is a left inverse for $\qhom$, that is $\eval\circ \qhom = \id_{\subgroupParallel}$

\item[\rm2)]
$\qhom(\subgroupMeridian) \subset \ker\eval$

\item[\rm3)]
$\jO\bigl(\pi_1(\DTC, \SfC)\bigr) \subset \ker\eval$.
\end{enumerate}
\end{proposition}
\begin{corollary}\label{cor:splitting_pi1DS}
\begin{enumerate}
\item[\rm a)]
The map $\theta: \ker\eval \times \subgroupParallel \longrightarrow \pi_1(\DT,\Sf)$ defined by
$\theta(\omega, l) = \omega \cdot \qhom(l)$ for $(\omega, l) \in \ker\eval \times \subgroupParallel$, is a groups isomorphism.

\item[\rm b)]
The following sequence is exact: \
$1 \longrightarrow \subgroupMeridian \xrightarrow{~~\qhom~~} \ker\eval \xrightarrow{~~\partial~~} \pi_0\Sf \longrightarrow 1$.
\end{enumerate}
\end{corollary}
\begin{proof}
Statement a) follows from 1) of Proposition~\ref{pr:hom_onto_L} and the fact that $\qhom(\subgroupParallel)$ is contained in the center of $\pi_1(\DT,\Sf)$, see (3) of Theorem~\ref{th:main:pi1Of}.
Statement b) is a direct consequence of statements 1) and 2) of Proposition~\ref{pr:hom_onto_L} and Lemma~\ref{lm:partial_splitting} applied to the lower exact sequence of Eq.~\eqref{equ:morphism_of_sequences}.
We leave the details for the reader.
\end{proof}

Due to 2) and 3) of Proposition~\ref{pr:hom_onto_L} and 3) of Corollary~\ref{cor:splitting_pi1DS} we see that diagram Eq.~\eqref{equ:morphism_of_sequences} reduces to the following one:
\begin{equation}\label{equ:morphism_of_sequences_reduced}
\begin{CD}
1 @>{}>> 1 @>>> \pi_1(\DTC,\SfC) @>{\partial_{\curveMeridian}}>{\cong}> \pi_0\SfC @>{}>> 1 \\
&& @V{}VV @V{\jO}VV @V{\jZ}VV \\
1 @>{}>> \subgroupMeridian @>{\qhom}>> \ker\eval @>{\partial}>> \pi_0\Sf @>{}>> 1
\end{CD}
\end{equation}
Thus to complete Theorem~\ref{th:main:pi1Of} it suffices to prove that the middle vertical arrow, $\jO$, in Eq.~\eqref{equ:morphism_of_sequences_reduced} is an isomorphism between $\pi_1(\DTC, \SfC)$ and $\ker\eval$.
As $\subgroupParallel \cong \ZZZ \cong \pi_1 S^1$ we will then get the required isomorphism
\[
\pi_1(\DTC, \SfC) \times \ZZZ  \ \cong \ \ker\eval \times \subgroupParallel \ \xrightarrow{~~\theta~~} \ \pi_1(\DT,\Sf).
\]

To show that $\jO: \pi_1(\DTC, \SfC) \to \ker\eval$ is an isomorphism notice that $\partial_{\curveMeridian}$ is also an isomorphism by (4) of Theorem~\ref{th:fibration_DMX_Of}.
Therefore from the latter diagram Eq.~\eqref{equ:morphism_of_sequences_reduced} we get the following one:
\[
\begin{CD}
1 @>{}>> \ker \jZ @>>> \pi_0\SfC @>{\jZ}>> \pi_0\Sf  \\
&& @V{}VV @V{\jO \circ \partial_{\curveMeridian}^{-1}}VV @| \\
1 @>{}>> \qhom(\subgroupMeridian) @>>> \ker\eval @>{\partial}>> \pi_0\Sf @>{}>> 1
\end{CD}
\]
\begin{proposition}\label{pr:j0_properties}
Homomorphism $\jZ:\pi_0\SfC \longrightarrow \pi_0\Sf$ is surjective, and the induced map \[ \jO \circ \partial_{\curveMeridian}^{-1}:\ker \jZ \longrightarrow q(\subgroupMeridian)\] is an isomorphism.
\end{proposition}

In other words, Proposition~\ref{pr:j0_properties} claims that we have the following morphism between short exact sequences:
\[
\begin{CD}
1 @>{}>> \ker \jZ @>>> \pi_0\SfC @>{\jZ}>> \pi_0\Sf  @>{}>> 1 \\
&& @V{\cong}VV @V{\jO \circ \partial_{\curveMeridian}^{-1}}VV @| \\
1 @>{}>> \qhom(\subgroupMeridian) @>>> \ker\eval @>{\partial}>> \pi_0\Sf @>{}>> 1
\end{CD}
\]
Since left and right vertical arrows are isomorphisms, it will follow from five lemma, \cite[\S~2.1]{Hatcher:AlgTop:2002}, that $\jO \circ \partial_{\curveMeridian}^{-1}$ is an isomorphism as well.
This completes Theorem~\ref{th:main:pi1Of} modulo Propositions~\ref{pr:hom_onto_L} and~\ref{pr:j0_properties}.
The next two sections are devoted to the proof of those propositions.

\begin{remark}\rm
It easily follows from statement a) of Corollary~\ref{cor:splitting_pi1DS} that the map $\kappa:\OfC\times S^1 \to \Of$ defined by:
\begin{equation}\label{equ:equivalence}
\kappa(g,t) = g\circ \flowParallel_t
\end{equation}
is a homotopy equivalence.
We leave the details for the reader.
\end{remark}

\section{Proof of Proposition~\ref{pr:hom_onto_L}}
Existence of $\eval$ is guaranteed by statement (e) of the following lemma.
\begin{lemma}\label{lm:comm_diagramm}
Suppose $\Sf$ trivially acts of $\curveMeridian$.
Then there is a commutative diagram:
\begin{equation}\label{equ:comm_diagr}
\xymatrix{
                                 && \pi_1\DT \ar[d]_{\qhom}              && \pi_1\Torus \ar[ll]_-{\torInDiff}^-{\cong}  \ar@/^/[d]^-{r} \\
\pi_1(\DTC, \SfC) \ar[rr]^-{\jO} && \pi_1(\DT, \Sf) \ar[rr]^-{\unwind}  && \pi_1(\Torus, \curveMeridian) \ar@/^/[u]^-{s} \cong \ZZZ
}
\end{equation}
in which 
\begin{enumerate}
\item[\rm(a)]
$r\circ s$ is the identity isomorphism of $\pi_1(\Torus, \curveMeridian)$;
\item[\rm(b)]
$\torInDiff \circ s$ is an isomorphism of $\pi_1(\Torus,\curveMeridian)$ onto $\subgroupParallel$;
\item[\rm(c)]
$\unwind$ is surjective;
\item[\rm(d)]
$\qhom(\subgroupMeridian) \subset \ker\unwind$;
\item[\rm(e)]
$\jO(\pi_1(\DTC, \SfC)) \ \subset \ \ker\unwind$.
\item[\rm(f)]
The following composition 
\[
\eval = \torInDiff \circ s \circ \unwind: \pi_1(\DTC, \SfC) \longrightarrow \subgroupParallel
\]
satisfies the statement of Proposition~\ref{pr:hom_onto_L}.
\end{enumerate}
\end{lemma}
\begin{proof}
We need only to define the map $\unwind$, since $\qhom$ appears in Eq.~\eqref{equ:exact_seq_of_pair_DMX_SfX}, and  $r$, $s$, and $\torInDiff$ are described in \S\ref{sect:some_constructions}.

Let $\omega:(I, \partial I, 0) \longrightarrow (\DT, \Sf, \id_{\Torus})$ be a continuous map.
In particular, $\omega(1)$ belongs to $\Sf$.
By assumption $\Sf$ trivially acts on $\curveMeridian$, whence $\omega(1)(\curveMeridian) = \curveMeridian$.
Consider the following path
\[
\omega_{e}: I \to \Torus,
\qquad
\omega_{e}(t) = \omega(t)(e).
\]
Then $\omega_e(0) = e$ and $\omega_e(1) \in \curveMeridian$.
Therefore $\omega_e \in C\bigl( (I, \partial I, 0), (\Torus,\curveMeridian, e)$ and we put by definition
\[ \unwind(\omega) = \omega_e. \]
The map $\unwind:\pi_1(\DT,\Sf) \to \pi_1(\Torus,\curveMeridian)$ in Eq.~\eqref{equ:comm_diagr} is the induced mapping of the corresponding sets of homotopy classes.
It is easy to verify that $\unwind$ is in fact a group homomorphism.

\smallskip

{\bf Commutativity of diagram Eq.~\eqref{lm:comm_diagramm}.}
Notice that the groups in right rectangle of Eq.~\eqref{lm:comm_diagramm} are just the sets of path components of the corresponding spaces from the following diagram:
\begin{equation}\label{equ:diagram_maps_of_pairs}
\xymatrix{
C\bigl( (I, \partial I), (\DT, \id_{\Torus}) \bigr) \ar[d]_{\qhom}               && C\bigl( (I, \partial I), (\Torus,e) \bigr) \ar[ll]_-{\torInDiff}  \ar[d]^{r} \\
C\bigl( (I, \partial I, 0), (\DT, \Sf, \id_{\Torus}) \bigr) \ar[rr]^-{\unwind}  &&  C\bigl( (I, \partial I, 0), (\Torus,\curveMeridian, e) 
}
\end{equation}
Notice that the maps $\qhom$ and $r$ here are just natural inclusions.
It suffices to prove commutativity of diagram~\eqref{equ:diagram_maps_of_pairs}.

Let $\omega=(\alpha,\beta):(I,\partial I) \longrightarrow (\Torus,e)$ be a representative of some loop in $\pi_1\Torus$, where $\alpha$ and $\beta$ are coordinate functions of $\omega$.
Then 
\[
\qhom\circ \torInDiff(\omega)(t)(x,y) = 
\bigl(x+\alpha(t) \ \mathrm{mod} \ 1, \ \ y+\beta(t) \ \mathrm{mod} \ 1\bigr).
\]
whence
\[
\unwind \circ q\circ \torInDiff(\omega)(t) =
\qhom\circ \torInDiff(\omega)(t)(0,0)  = 
\bigl(\alpha(t), \beta(t) \bigr)  = \omega(t).
\]
But $r(\omega)(t) = \omega(t)$ as well, whence $\unwind \circ \qhom\circ \torInDiff(\omega) = r(\omega)$.
Thus diagram Eq.~\eqref{lm:comm_diagramm} is commutative.

{\bf Property (a)} is already established, see remark just after Eq.~\eqref{equ:sect_of_r}.

\smallskip

{\bf Property (b).}
Since $\torInDiff\circ s ( r[\cycleParallel]) = \torInDiff[\cycleParallel] = \flowParallel$ and $\subgroupParallel$ is freely generated by $\flowParallel$, it follows that $\torInDiff\circ s$ isomorphically maps $\pi_1(\Torus,\curveMeridian)$ onto $\subgroupParallel$.

\smallskip

{\bf Property (c).}
By commutativity of Eq.~\eqref{equ:comm_diagr} we have that $r[\cycleParallel] = \unwind \circ \qhom \circ \torInDiff[\cycleParallel]$.
But $r[\cycleParallel]$ generates $\pi_1(\Torus,\cycleParallel)$, whence $\unwind$ is surjective.

\smallskip

{\bf Property (d).}
As $\flowMeridian$ generates $\subgroupMeridian$, it suffices to show that $\unwind \circ \qhom(\flowMeridian) = 0$.
We have that $[\cycleMeridian] \in k(\pi_1\curveMeridian)$.
Therefore it follows from exactness of sequence~\eqref{equ:exact_seq_T_C} that $r[\cycleMeridian] = 0 \in \pi_1(\Torus,\curveMeridian)$.
Hence
\[
\unwind \circ \qhom[\flowMeridian] = \unwind \circ \qhom\circ \torInDiff [\cycleMeridian] = r[\cycleMeridian] = 0.
\]
Thus $\qhom(\subgroupMeridian) \subset \ker\eval$;

\smallskip

{\bf Property (e).}
Let $\alpha \in \pi_1(\DTC, \SfC)$ and $\omega:(I,\partial I, 0) \longrightarrow (\DTC, \SfC, \id_{\Torus})$ a path representing $\alpha$.
Then $\jO(\alpha)$ is represented by the homotopy class of the map \[ i\circ\omega:(I,\partial I, 0) \longrightarrow (\DT, \Sf, \id_{\Torus}).\]
By assumption $\omega(t)\in\DTC$, i.e. it is fixed on $\curveMeridian$ for all $t\in I$.
In particular, since $e\in\curveMeridian$, we get that
\[
\unwind(\omega)(t) = \omega_e(t) = \omega(t)(e) = e.
\]
Thus $\unwind(\omega): I \to \Torus$ is a constant map, and so it represents a unit element of $\pi_1(\Torus,\curveParallel)$.
Therefore $\jO(\alpha)\in\ker\unwind$.

\smallskip

{\bf Property (f).}
By (b) and (c) $\torInDiff\circ s$ is an isomorphism, and $\unwind$ is surjective.
Hence $\eval = \torInDiff \circ s \circ \unwind$ is surjective as well, and $\ker\eval = \ker\unwind$.
Therefore statements 2) and 3) of Proposition~\ref{pr:hom_onto_L} follow from (d) and (e) respectively.

To prove 1) notice that
\[ 
\eval \circ \qhom (\flowParallel) 
= (\torInDiff \circ s \circ \unwind) \circ \qhom \bigl(\torInDiff [\cycleParallel]\bigr) 
= \torInDiff \circ s \circ (\unwind \circ \qhom \circ \torInDiff) [\cycleParallel]
= \torInDiff \circ s \circ r [\cycleParallel] =  \torInDiff[\cycleParallel] = \flowParallel.
\]
Hence $\eval\circ \qhom = \id_{\subgroupParallel}$.
Lemma~\ref{lm:comm_diagramm} and Proposition~\ref{pr:hom_onto_L} are completed.
\end{proof}

\section{Proof of Proposition~\ref{pr:j0_properties}}
\subsection{Image of $\jZ$}
Surjectivity of $\jZ$ is guaranteed by the following lemma.
\begin{lemma}\label{pr:im_j0}
Suppose $\Sf$ trivially acts on $\curveMeridian$, i.e. $h(\curveMeridian) = \curveMeridian$ for all $h\in\Sf$.
Then the induced homomorphism $\jZ:\pi_0\SfC\to\pi_0\Sf$ is surjective.
\end{lemma}
\begin{proof}
We should prove that each $h\in\Sf\equiv\Stabf \cap\DiffIdT$ is isotopic in $\Sf$ to a diffeomorphism $g$ belonging to $\SfC\equiv \Stabf \cap \DiffIdTC$.

Indeed, by assumption $h(\curveMeridian)=\curveMeridian$.
Hence $h$ is isotopic in $\Sf$ to a diffeomorphism $g$ fixed on $\curveMeridian$, i.e. $g\in\Sf \cap \DiffTC$.
But due to Lemma~\ref{lm:h_C_id__h_in_DidTC}
\[
\Sf \cap \DiffTC =
\Stabf \cap \DiffIdT \cap \DiffTC \stackrel{\eqref{equ:DiffIdT_cap_DiffTC__DiffIdTC}}{=\!=\!=\!=} 
\Stabf \cap \DiffIdTC = \SfC,
\]
whence $g\in \SfC$.
Thus $\jZ$ is epimorphism.
\end{proof}

\subsection{Kernel of $\jZ$}\label{subsec:proof:pr:ker_j0}
Notice that the kernel of $\jZ:\pi_0\SfC\to\pi_0\Sf$ consists of isotopy classes of diffeomorphisms in $\SfC$ isotopic to $\id_{\Torus}$ by $f$-preserving isotopy, however such an isotopy should not necessarily be fixed on $\curveMeridian$.
In other words, if we denote 
\[
 \kerjo \ := \ \Sidf \cap\DTC \ = \
 \StabIdf \ \cap \ \Diff(\Torus,\curveMeridian),
\]
then
\begin{equation}\label{equ:kerj0__pi0_StabIdf_DiffIdTC}
\ker\jZ \ = \ \pi_0 \kerjo. 
\end{equation}
Also notice that $\SidfC$ is the identity path component of $\kerjo$, whence
\[
\ker \jZ \ = \ \pi_0\kerjo \ = \ \kerjo /\SidfC. 
\]

For the proof of Proposition~\ref{pr:j0_properties} we will first establish in Lemma~\ref{lm:ker_j0} that $\ker\jZ\cong \ZZZ$ and then show in Lemma~\ref{lm:iso_kerjZ_M} that $\jO \circ \partial_{\curveMeridian}^{-1}$ yields an isomorphism of $\ker\jZ$ onto $\subgroupMeridian$.

Since $\kerjo \ := \ \Sidf \cap \DTC  \ \subset \ \Sidf$, it follows from Lemma~\ref{lm:shift_functions} that for every $h\in\kerjo$ there exists a unique smooth function $\delta\in C^{\infty} (\Torus)$ such that $h = \flow_{\delta}$.

\begin{lemma}\label{lm:ker_j0}
Suppose the flow $\flow$ satisfies the relation Eq.~\eqref{equ:relation_flowMer_flowHam}.
Then for each $h= \flow_{\delta} \in \kerjo$ the function $\delta$ is constant on $\curveMeridian$ and its value on $\curveMeridian$ is integer.
Define a map $\eta: \kerjo \longrightarrow \ZZZ$ by
\[
\eta(h) = \delta|_{\curveMeridian},
\]
for $h = \flow_{\delta} \in \kerjo$.
Then $\eta$ is a surjective homomorphism with $\ker\eta = \SidfC$.
In particular $\eta$ yields an isomorphism $\ker \jZ = \pi_0\kerjo = \kerjo /\SidfC \cong \ZZZ.$
\end{lemma}
\begin{proof}
We will regard $\flowMeridian$ as a flow on $\Torus$.
Notice that $\curveMeridian$ is a closed trajectory of both flows $\flow$ and $\flowMeridian$ and its period with respect to $\flowMeridian$ is equal to $1$.
Therefore the relation Eq.~\eqref{equ:relation_flowMer_flowHam} implies that the period of $\curveMeridian$ with respect to $\flow$ also equals $1$.

Since $h\in\kerjo =  \Sidf \cap \DTC$ is fixed on $\curveMeridian$, that is $y= h(y) = \flow(y,\delta(y))$ for all $y\in \curveMeridian$, it follows that the value of $\delta(y)$ a multiple of the period $\theta=1$.
Thus for every $y\in\curveMeridian$ there exists $n_y\in\ZZZ$ such that $\delta(y)= n_y$.
But $\delta$ is continuous, whence the mapping $y\mapsto n_y = \delta(y)$ is a continuous function $\curveMeridian \to \ZZZ$.
Therefore this function is constant, i.e. $\delta|_{\curveMeridian} = n$ for some $n\in\ZZZ$.

\smallskip

We should prove that $\eta$ has the desired properties.

\smallskip

{\bf Step 1.} {\em $\eta$ is a homomorphism.}

Let $h_i = \flow_{\delta_{i}} \in \kerjo$ for $i=0,1$ such that $\delta_{i}|_{\curveMeridian} = n_i$ for certain $n_i\in\ZZZ$.
Define the function $\delta = \delta_1 \circ h_0 + \delta_0$.
Since $h_0$ is fixed on $\curveMeridian$, we have that for each $z\in\curveMeridian$
\[ 
\delta(z) \ =  \
\delta_1 \circ h_0(z) + \delta_0(z)= 
\delta_1 (z) + \delta_0(z) = n_1 + n_0.
\]
Moreover, by \cite[Eq.~(8)]{Maksymenko:TA:2003},
\[ h_1 \circ h_0 = \flow_{\delta},\]
whence
\[
\eta(h_1 \circ h_0) = 
\delta|_{\curveMeridian} = n_1 + n_0 = \delta_1 |_{\curveMeridian} + \delta_0|_{\curveMeridian} =  \eta(h_1) + \eta(h_0),
\]
and so $\eta$ is a homomorphism.

\smallskip

{\bf Step 2.} {\em $\eta$ is surjective.}

It suffices to construct $g\in\kerjo$ with $\eta(g)=-1$.
The choice of $-1$ will simplify further exposition, however we could also construct $g$ satisfying $\eta(g)=+1$.

Let $J_{\eps}$ be the same as in Eq.~\eqref{equ:relation_flowMer_flowHam}, and $\beta: J_{\eps}\to [0,1]$ be a $C^{\infty}$-function such that 
\[
\beta(z) = 
\begin{cases}
-1, & |z|< \eps/3,\\
0 & |z|> 2\eps/3.
\end{cases}
\]
Define another function $\sigma: \Torus \to \RRR$ by
\[
\sigma(x,y) = 
\begin{cases}
\beta(x),  & x\in J_{\eps}, \\
0,         & x\in S^1\setminus J_{\eps},
\end{cases}
\]
and let $g = \flow_{\sigma}$, so $g$ is a map $\Torus\to\Torus$ defined by
\[ g(x,y) = \flow(x,y,\sigma(x,y)). \]
Since $\sigma=0$ outside $J_{\eps}\times S^1$, it follows from  Eq.~\eqref{equ:relation_flowMer_flowHam} that $g = \flowMeridian_{\sigma}$, i.e.
\begin{equation}\label{equ:gen_pi0SfC_via_flowM}
g(x,y) = \flowMeridian(x,y,\sigma(x,y)).
\end{equation}
Moreover, as periods of all points with respect to $\flowMeridian$ are equal to $1$, we can also write $g = \flowMeridian_{\sigma+n}$ for any $n\in\ZZZ$, i.e.
\[ g(x,y) = \flowMeridian(x,y,\sigma(x,y)+n). \]
Thus
\[ g = \flow_{\sigma} = \flowMeridian_{\sigma} = \flowMeridian_{\sigma+n} \]
for all $n\in\ZZZ$.
The following properties of $g$ can easily be verified:
\begin{itemize}
\item[(i)]
$g$ is fixed on $J_{\eps/3}\times S^1$ and outsize $J_{\eps}\times S^1$;

\item[(ii)]
$g \in \StabIdf$ by Lemma~\ref{lm:shift_functions}, since $g = \flow_{\sigma}$;

\item[(iii)]
the following family of maps $\mathbf{G}_t = \flowMeridian_{t(\sigma+1)}$, $t\in I$, is an isotopy between 
\begin{align}\label{equ:isot_g_id}
\mathbf{G}_0 &= \flowMeridian_{0} = \id_{\Torus}, &
\mathbf{G}_1 &= \flowMeridian_{\sigma+1} = g,
\end{align}
see Figure~\ref{fig:g_and_Gt}.
This isotopy is fixed on $\curveMeridian$, as $t(\sigma+1) = 0$ on $\curveMeridian$.
Hence $g\in\DTC$ as well.
\end{itemize}
Thus by (ii) and (iii) $g\in \Stabf \cap \DTC = \kerjo$.
It remains to note that $\eta(g) = \sigma|_{C} = -1$, and so $\eta$ is surjective.

\begin{figure}[h]
\center{\includegraphics[width=6cm]{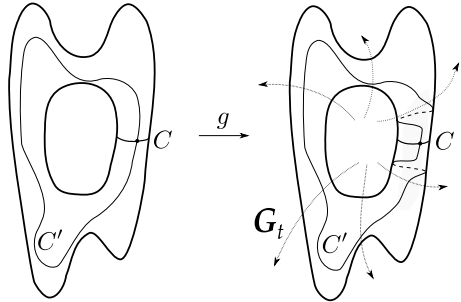}}
\caption{Diffeomorphism $g$ and the isotopy $\mathbf{G}_t$}
\label{fig:g_and_Gt}
\end{figure}

{\bf Step 3.} {$\ker \eta = \SidfC$.}

Suppose $h=\flow_{\delta}\in\ker\eta$, so $\eta(h) = \delta|_{\curveMeridian} =0$.
Then by Lemma~\ref{lm:shift_functions} an isotopy between $h$ and $\id_{\Torus}$ can be given by $g_t = \flow_{t\delta}$, $t\in[0,1]$.
Notice that $t\delta|_{\curveMeridian} =0$ as well, whence $\{g_t\}$ is also fixed on $\curveMeridian$.
Therefore $h\in \SidfC$.

Conversely, let $h \in \SidfC$, so $h$ is isotopic to $\id_{\Torus}$ in $\Sidf$ via an isotopy $\{h_t\}_{t\in[0,1]}$ fixed on $\curveMeridian$ and such $h_0 = \id_{\Torus}$ and $h_1 = h$.
Then $h_t = \flow_{\delta_t}$, $t\in[0,1]$, for some smooth function $\delta_t:\Torus\to\RRR$.
Since $\curveMeridian$ is a non-fixed trajectory of $\flow$, it follows from~\cite[Theorem~25]{Maksymenko:TA:2003}, that the values of $\delta_t$ on $\curveMeridian$ continuously depend on $t$.
But each $\delta_t$ takes a constant integer value on $\curveMeridian$, and $\id_{\Torus} = \flow_{0}$, whence 
\[
\eta(h) = \delta_1|_{\curveMeridian} =  \delta_t|_{\curveMeridian} = \delta_0|_{\curveMeridian} = 0,
\]
that is $h\in\ker\eta$.
\end{proof}

\subsection{Inverse of boundary isomorphism $\partial_{\curveMeridian}^{-1}: \pi_0\SfC \longrightarrow \pi_1(\DTC,\SfC)$}
Thus we have that both $\ker\jZ$ and $\subgroupMeridian$ are isomorphic to $\ZZZ$.
By Lemma~\ref{lm:ker_j0} $\ker\jZ$ is generated by the homotopy class of the diffeomorphism 
\[ g = \flow_{\sigma} = \flowMeridian_{\sigma} \in \Stabf \cap \DiffId(\Torus,\curveMeridian) \subset \SfC\]
defined by Eq.~\eqref{equ:gen_pi0SfC_via_flowM} and satisfying $\eta(g)=-1$.

On the other hand, $q(\subgroupMeridian)$ is generated by the homotopy class of the following map:
\begin{align}\label{equ:M_as_a_map_of_triples}
& \qhom(\flowMeridian): (I, \partial I, 0) \longrightarrow (\DT, \Sf, \id_{\Torus}),
&
\qhom(\flowMeridian)(t)& = \flowMeridian_{t}
\end{align}
Therefore in order to complete Proposition~\ref{pr:j0_properties} it suffices to establish the following lemma. 

\begin{lemma}\label{lm:iso_kerjZ_M}
$\jO \circ \partial_{\curveMeridian}^{-1}[g] = [\qhom(\flowMeridian)]$.
Hence $\jO \circ \partial_{\curveMeridian}$ isomorphically maps $\ker\jZ$ onto $\qhom(\subgroupMeridian)$.
\end{lemma}
\begin{proof}
Recall that the boundary homomorphism $\partial_{\curveMeridian}:\pi_1(\DTC,\SfC) \longrightarrow \pi_0\SfC$ is defined as follows: if $\alpha\in\pi_1(\DTC,\SfC)$ and  $\omega: (I,\partial I, 0) \to (\DTC,\SfC,\id_{\Torus})$ is a representative of $\alpha$, then
\[ \partial_{\curveMeridian}(\alpha)  =  [\omega(1)] \in\pi_0\SfC.\]
Now let $\mathbf{G}_t = \flowMeridian_{t(\sigma+1)}$ be an isotopy between $\mathbf{G}(0)=\id_{\Torus}$ and $\mathbf{G}(1)=g$ fixed on $\curveMeridian$, see Eq.~\eqref{equ:isot_g_id}.
Regard it as a map of triples $\mathbf{G}: (I, \partial I, 0) \longrightarrow (\DTC, \SfC, \id_{\Torus})$.

Then $\partial([\mathbf{G}]) = [\mathbf{G}(1)] = [g]$, and so \[\partial_{\curveMeridian}^{-1}[g] = [\mathbf{G}].\]
As $\partial_{\curveMeridian}$ is an \textit{isomorphism}, $\partial_{\curveMeridian}^{-1}[g]$ \textit{does not depend} on a particular choice of such an isotopy $\mathbf{G}$.
Furthermore, $\jO\circ \partial_{\curveMeridian}^{-1}[g]$ is a homotopy class of $\mathbf{G}$ regarded as a map
\begin{align}\label{equ:G_as_a_map_of_triples}
\mathbf{G}:& (I, \partial I, 0) \longrightarrow (\DT, \Sf, \id_{\Torus}),
&
\mathbf{G}(t)& = \flowMeridian_{t(\sigma+1)}.
\end{align}
Therefore it remains to show that $[\mathbf{G}] = [\qhom(\flowMeridian)]$, that is the maps~\eqref{equ:M_as_a_map_of_triples} and~\eqref{equ:G_as_a_map_of_triples} are homotopic as maps of triples.

In fact the homotopy between them can be defined as follows:
\begin{align*}
&\mathbf{H}: (I, \partial I, 0) \times I \longrightarrow (\DT, \Sf, \id_{\Torus}),
&
\mathbf{H}(t,s) &= \flowMeridian_{t(s\sigma+1)}.
\end{align*}

1) First we verify that \textit{$\mathbf{H}(t,s)$ is a diffeomorphism for all $t,s\in I$}. 
As $g = \flowMeridian_{\sigma}$ is a diffeomorhism, it follows from Lemma~\ref{lm:shift_functions} that $\flowMeridian_{s\sigma} = \flowMeridian_{s\sigma+1}$ is also a diffeomorphism for all $s\in I$.
But then by the same lemma $\mathbf{H}(t,s) = \flowMeridian_{t(s\sigma+1)}$ is a diffeomorphism for all $t,s\in I$ as well.

2) Now let us show that $\mathbf{H}$ is a homotopy of maps of triples.
Indeed, for each $s\in I$ we have that 
$\mathbf{H}_{0,s} = \flowMeridian_{0} = \id_{\Torus}$, and  $\mathbf{H}_{1,s} = \flowMeridian_{s\sigma+1} = \flowMeridian_{s\sigma} \in \Sf$.

3) Finally, $\mathbf{H}_{t,0} = \flowMeridian_{t} = q(\flowMeridian)(t)$, and $\mathbf{H}_{t,1} = \flowMeridian_{t(\sigma+1)} = \mathbf{G}(t)$ for all $t\in I$.
Thus $\mathbf{H}$ is a homotopy between $\qhom(\flowMeridian)$ and $\mathbf{G}$.
Lemma~\ref{lm:iso_kerjZ_M} and Proposition~\ref{pr:j0_properties} are completed.
\end{proof}


\def\cprime{$'$}
\providecommand{\bysame}{\leavevmode\hbox to3em{\hrulefill}\thinspace}
\providecommand{\MR}{\relax\ifhmode\unskip\space\fi MR }
\providecommand{\MRhref}[2]{%
  \href{http://www.ams.org/mathscinet-getitem?mr=#1}{#2}
}
\providecommand{\href}[2]{#2}

%
%
%
%
\end{document}